\documentclass[12pt,reqno]{amsart}
\usepackage{etoolbox}
\makeatletter
\patchcmd\maketitle
{\uppercasenonmath\shorttitle}
{}
{}{}
\patchcmd\maketitle
{\@nx\MakeUppercase{\the\toks@}}
{\the\toks@}
{}
{}{}
\patchcmd\@settitle{\uppercasenonmath\@title}{\Large}{}{}
\patchcmd\@setauthors
{\MakeUppercase{\authors}}
{\authors}
{}{}
\makeatother
\usepackage{amsmath,amssymb,amsthm, amsfonts}
\usepackage{color}
\usepackage{xcolor}
\usepackage{url}
\usepackage{tikz-cd}
\usepackage{combelow}
\usepackage{enumerate}
\input{mathrsfs.sty}
\usepackage[utf8]{inputenc}
\usepackage[T1]{fontenc}
\newtheorem{theorem}{Theorem}[section]
\newtheorem{question}{Question}[section]
 
\newtheorem{definition}{Definition}[section]

\newtheorem{corollary}{Corollary}[section]

\newtheorem{lemma}{Lemma}[section]
\newtheorem{remark}{Remark}[section]
\newtheorem{example}{Example}[section]
\numberwithin{equation}{section}
\newcommand\norm[1]{\left\lVert#1\right\rVert}
\newcommand\aps[1]{\left\lvert#1\right\rvert}
\newcommand\skal[2]{\left\langle #1,#2\right\rangle}

\usepackage[colorlinks=true]{hyperref}
\hypersetup{urlcolor=blue, citecolor=red , linkcolor= blue}
\usepackage{cite}

\begin{document}
		%------------------------
		\title[On roots of normal operators and extensions of Ando's Theorem]{On roots of normal operators and extensions of Ando's Theorem}
		
		
		\author[H. Stankovi\'c]{Hranislav Stankovi\'c}
		\address{Faculty of Electronic Engineering, University of Ni\v s, Aleksandra Medvedeva 4, Ni\v s, Serbia
		}
		\email{\url{hranislav.stankovic@elfak.ni.ac.rs}}

		\author[C. Kubrusly]{Carlos Kubrusly}
		\address{Catholic University of Rio de Janeiro, 22453-900, Rio de Janeiro, RJ, 
			Brazil
		}
		\email{\url{carlos@ele.puc-rio.br}}

		%------------------------

		\date{\today}
		%\date{June 24, 2019}
		
		\maketitle
		%------------------------
		%----------\thispagestyle{empty}

		\begin{abstract}
			In this paper, we extend Ando's theorem on paranormal operators, which states that if \( T \in \mathfrak{B}(\mathcal{H}) \) is a paranormal operator and there exists \( n \in \mathbb{N} \) such that \( T^n \) is normal, then \( T \) is normal. We generalize this result to the broader classes of \( k \)-paranormal operators and absolute-\( k \)-paranormal operators. Furthermore, in the case of a separable Hilbert space $\mathcal{H}$, we show that if \( T \in \mathfrak{B}(\mathcal{H}) \) is a \( k \)-quasi-paranormal operator for some \( k \in \mathbb{N} \), and there exists \( n \in \mathbb{N} \) such that \( T^n \) is normal, then \( T \) decomposes as \( T = T' \oplus T'' \), where \( T' \) is normal and \( T'' \) is nilpotent of nil-index at most \( \min\{n,k+1\} \), with either summand potentially absent. 
		\end{abstract}
		
		\bigskip 
		
		\section{Introduction}
		
		\bigskip 
		
		Let $(\mathcal{H},\skal{\,}{\,})$ be a complex Hilbert space, and let $\mathfrak{B}(\mathcal{H})$ denote the algebra of bounded linear operators on $\mathcal{H}$. 
		The null space and the range of an operator $T \in \mathfrak{B}(\mathcal{H})$ are denoted by $\mathcal{N}(T)$ and $\mathcal{R}(T)$, respectively. The adjoint of $T$ is denoted by $T^*$. A closed subspace $\mathcal{L}\subseteq\mathcal{H}$ is a reducing subspace of an operator $T\in\mathfrak{B}(\mathcal{H})$ if it is invariant under both $T$ and $T^*$, i.e., if $T(\mathcal{L})\subseteq \mathcal{L}$ and $T^*(\mathcal{L})\subseteq \mathcal{L}$. An operator $T$  is said to be nilpotent if there exists $n\in\mathbb{N}$ such that $T^n=0$.  The smallest such $n$ (if it exists) is called the nil-index of a nilpotent operator $T$.
		An operator $T$ is said to be positive, in notation $T\geq 0$, if $\skal{Tx}{x}\geq 0$ for all $x\in\mathcal{H}$,  and self-adjoint (Hermitian) if $T=T^*$.
		
		An operator $T$ is said to be normal if $T^*T = TT^*$, while it is said to be pure if it has no nonzero reducing subspace on which
		it is normal. It is well known  that each $T\in\mathfrak{B}(\mathcal{H})$ can be represented  in a unique way as $T=T_n\oplus T_p$,
		where $T_n$ is normal and  $T_p$ is a pure operator (see \cite[Proposition 2.1]{Conway91}).
		The theory of normal operators has been extensively developed, primarily due to the applicability of the Spectral Theorem, which plays a central role in their analysis.
		
		Because of their fundamental importance in operator theory and quantum mechanics, numerous generalizations of normal operators have been introduced over the years. Some of the most prominent classes include:
		
		\begin{itemize}
			\item quasinormal operators: $T$ commutes with $T^*T$, i.e., $TT^*T = T^*T^2$;
			\item subnormal operators: there exist a Hilbert space $\mathcal{L}$ and a normal operator $N \in \mathfrak{B}(\mathcal{H}\oplus\mathcal{L})$ such that
			\[
			N = \begin{bmatrix} T & * \\ 0 & * \end{bmatrix} :
			\begin{pmatrix}
				\mathcal{H} \\
				\mathcal{L}
			\end{pmatrix}
			\to
			\begin{pmatrix}
				\mathcal{H} \\
				\mathcal{L}
			\end{pmatrix};
			\]
			\item hyponormal operators: $TT^* \leq T^*T$;
			\item $p$-hyponormal operators: $(TT^*)^p \leq (T^*T)^p$ for some $0<p\leq 1$;
			\item class $A$ operators: $T^*T \leq \left({T^*}^2 T^2\right)^{1/2}$;
			\item paranormal operators: $\|Tx\|^2 \leq \|T^2x\|\norm{x}$ for all $x \in \mathcal{H}$;
			\item normaloid operators: $r(T) = \|T\|$, where $r(T)$ denotes the spectral radius of $T$. Equivalently, $\norm{T^n}=\norm{T}^n$ for each $n\in\mathbb{N}$.
		\end{itemize}
		
		It is well known that the following (strict) inclusions hold:
		\begin{equation}\label{eq:inclusion_chain}
			\begin{split}
				\text{normal}
				&\subset \text{quasinormal}
				\subset \text{subnormal}
				\subset \text{hyponormal} \\
				&\subset \text{$p$-hyponormal}
				\subset \text{class $A$}
				\subset \text{paranormal} \\
				&\subset \text{normaloid}.
			\end{split}
		\end{equation}

		The classes of subnormal and hyponormal operators were introduced by Halmos in \cite{Halmos50}, while the study of quasinormal operators was first conducted by Brown in \cite{Brown53}. The class of $p$-hyponormal operators was defined as an extension of hyponormal operators in \cite{Xia83} and has since been investigated by many authors; see, for instance, \cite{Aluthge90, Aluthge96}. 
		
		The concepts of paranormal operators and operators of class $A$ were introduced by Istr\v a\cb{t}escu in \cite{Istratescu67} and by Furuta, Ito, and Yamazaki in \cite{FurutaItoYamazaki98}, respectively. Paranormal operators have garnered considerable attention, particularly in view of the results in \cite{Furuta67}, \cite{IstratescuIstratescu67}, and \cite{IstratescuSaitoYoshino66}. For an insightful comparison between paranormal operators and class $A$ operators, see \cite{Ito99}.
		
		For more details on the mentioned classes, we refer the reader to \cite{Furuta01, Kubrusly03}.

		\medskip 
		Several generalizations of the class of paranormal operators have been introduced, as well. More precisely, the following classes have emerged:

		\begin{definition}\label{eq:main_def}
			Let $k\in\mathbb{N}\cup\{0\}$. An operator \( T \in \mathfrak{B}(\mathcal{H}) \) is called 
			
			\begin{itemize}
				\item \( k \)-quasi-paranormal if  
				\begin{equation}\label{eq:quasi_paranormal_def}
					\|T^{k+1}x\|^2 \leq \|T^{k+2}x\|\|T^k x\|, \quad x \in \mathcal{H};
				\end{equation}
				\item \( k \)-paranormal if  
				\begin{equation}\label{eq:k-para_def}
					\norm{Tx}^{k+1}\leq\norm{T^{k+1}x}\norm{x}^k,\quad x\in\mathcal{H};
				\end{equation}
				\item absolute-$k$-paranormal  if  
				\begin{equation}\label{eq:abs-k-para_def}
					\norm{Tx}^{k+1}\leq\norm{|T|^kTx}\norm{x}^k,\quad x\in\mathcal{H}.
				\end{equation}
			\end{itemize}	
		\end{definition}
		Using the fact that $\norm{|T|Tx}=\norm{T^2x}$, for any $T\in\mathfrak{B}(\mathcal{H})$ and $x\in\mathcal{H}$, we note that
		\begin{align*}
			\text{paranormal}&\equiv \text{0-quasi-paranormal}\\&\equiv \text{\( 1 \)-paranormal}\\&\equiv \text{absolute-\( 1 \)-paranormal}.
		\end{align*} 
		
		The class of \( k \)-quasi-paranormal operators was introduced in \cite{GaoLi14}. In the special case when \( k = 1 \), these operators are known as quasi-paranormal, a class already studied in \cite{HanNa13}. The notion of \( k \)-paranormal operators was first mentioned in \cite{IstratescuIstratescu67}, and later studied in more detail in \cite{DuggalKubrusly11, KubruslyDuggal10}. The class of absolute-\( k \)-paranormal operators was defined in \cite{FurutaItoYamazaki98}, where some fundamental properties were also established.
		
		It follows from \cite[Lemma 2]{Furuta67} and \cite[Theorem 1]{IstratescuIstratescu67} (cf. \cite[Proposition 1]{KubruslyDuggal10}) that
		\begin{equation}\label{eq:k-para_chain}
			\text{normal}\subset \cdots\subset 	\text{paranormal} \subset \text{\( k \)-paranormal} \subset \text{normaloid}.
		\end{equation}
		Similarly, \cite[Theorem 2]{FurutaItoYamazaki98} and \cite[Theorem 5]{FurutaItoYamazaki98} yield
		
		\begin{equation}\label{eq:abs-k-para_chain}
			\text{normal}\subset \cdots\subset 	\text{paranormal} \subset \text{absolute-\( k \)-paranormal} \subset \text{normaloid}.
		\end{equation}
		
		On the other hand, the class of \( k \)-quasi-paranormal operators does not fit into the inclusion chain \eqref{eq:inclusion_chain}. Indeed, every nilpotent operator of order \( k +1\) (i.e., satisfying \( T^{k+1} = 0 \)), where $k\in\mathbb{N}$, is \( k \)-quasi-paranormal. However, since there exist nilpotent operators of any order \( k \geq 2 \) that are not normaloids, it follows that
		\[
		\text{\( k \)-quasi-paranormal} \not\subseteq \text{normaloid},\quad k\in\mathbb{N}.
		\]
		However, it is proved in \cite[Corollary 2.5]{GaoLi14} that if $T$ is $k$-quasi-paranormal for some $k\in\mathbb{N}$ and $\norm{T^{n+1}}=\norm{T^n}\norm{T}$ for some $n\geq k$, then $T$ is normaloid. On the other hand, the inclusion
		\[
		\text{paranormal} \subset \text{\( k \)-quasi-paranormal},\quad k\in\mathbb{N},
		\]
		clearly always holds. This follows directly by applying the definition of paranormality to the vector \( T^k x \) in place of \( x \). In fact, we have that
		\begin{equation*}
			\text{\( k \)-quasi-paranormal} \subset \text{\( (k+1) \)-quasi-paranormal}
		\end{equation*}
		for each $k\in\mathbb{N}$. 
		
		The aforementioned classes, along with the inclusion chains \eqref{eq:k-para_chain} and \eqref{eq:abs-k-para_chain}, will provide the basis for the subsequent discussion.

		\bigskip 
		
		\section{The $n$-th root problem}
		
		\bigskip 
		
		An intriguing problem in operator theory is the investigation of conditions under which certain operators are self-adjoint, normal, or belong to one of the other operator classes discussed above. This topic has attracted significant attention from many researchers, as evidenced by \cite{BeckPutnam56, Berberian70, Embry66, Embry70, Stampfli62, Stankovic23_roots, Stankovic23_factors, Stankovic24}. A closely related and natural question is identifying the conditions under which the reverse inclusions in \eqref{eq:inclusion_chain} hold.
		
		One line of investigation into these problems involves spectral properties of operators. For instance, in \cite{Putnam70}, it was shown that if a hyponormal operator has a spectrum of zero area, then the operator must be normal. This result has since been extended to other classes of operators (see \cite{ChoItoh95, Tanahashi04, Xia83}).
		
		Another fruitful approach involves examining powers (or polynomials) of operators that belong to a given class. More precisely, if \( T \in \mathfrak{K} \), where \( \mathfrak{K} \) is one of the classes appearing in \eqref{eq:inclusion_chain}, and if \( T^n \in \mathfrak{C} \) for some \( n \in \mathbb{N} \), where \( \mathfrak{C}\) is a subclass of \(\mathfrak{K} \), one may ask whether it follows that \( T \in \mathfrak{C} \).
		We refer to this as \emph{the $n$-th root problem (for the class $\mathfrak{C}$)}.

		\begin{figure}[h!]
			\begin{tikzpicture}
				
				% Outer ellipse for \mathfrak{C}
				\draw[thick] (0,0) ellipse (4cm and 2.5cm);
				\node at (-2.8,2.2) {\(\mathfrak{K}\)};
				
				% Inner ellipse for \mathfrak{C}'
				\draw[thick] (-0.5,0) ellipse (2.2cm and 1.5cm);
				\node at (-2.6,1.0) {\(\mathfrak{C}\)};
				
				% T^n inside \mathfrak{C}'
				\node[draw, rectangle, fill=green!20] (Tn) at (-1.2,-0.6) {\(T^n\)};
				
				% T outside \mathfrak{C}'
				\node[draw, rectangle, fill=blue!25] (T) at (2.5,0) {\(T\)};
				
				% T inside \mathfrak{C}'
				\node[draw, rectangle, fill=blue!5] (T') at (-0.2,0.2) {\(T\)};
				
				% Arrow from T to inside the inner ellipse (not to T^n)
				\draw[->, thick, dashed] (T) .. controls (1,1) .. (T');
				\node at (0.7,1.0) {?};
			\end{tikzpicture}
			\caption{The $n$-th root problem illustration.}
		\end{figure}
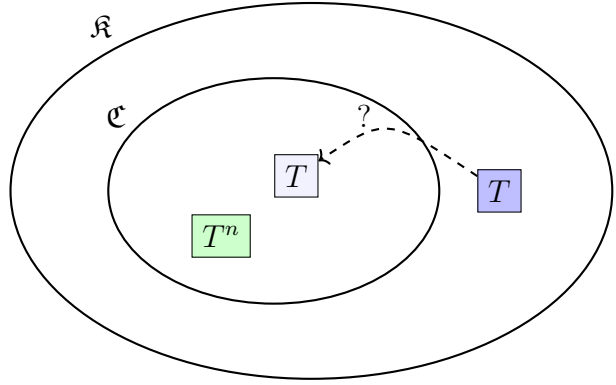

		A pioneering result in this direction is due to Stampfli \cite{Stampfli62}, who proved the following:

		\begin{theorem}\cite[Theorem 5]{Stampfli62}\label{thm:stampfli_thm}
			Let $T\in\mathfrak{B}(\mathcal{H})$ be a hyponormal operator. If there exists $n\in\mathbb{N}$ such that $T^n$ is normal, then $T$ is normal.
		\end{theorem}
		
		The original proof by Stampfli relied on a technique based on the Spectral Theorem. A more elementary argument, along with an extension, was provided in \cite{AluthgeWang99}. Specifically, the authors in \cite{AluthgeWang99} showed that for any \( p \)-hyponormal operator \( T \in \mathfrak{B}(\mathcal{H}) \) and any \( n \in \mathbb{N} \), the following inequality holds:
		\[
		((T^n)^* T^n)^{p/n} \geq (T^*T)^p \geq (T T^*)^p \geq (T^n (T^n)^*)^{p/n}.
		\]
		Therefore, if in addition \( T^n \) is normal for some \( n \in \mathbb{N} \), then all the inequalities become equalities:
		\[
		((T^n)^* T^n)^{p/n} = (T^*T)^p = (T T^*)^p = (T^n (T^n)^*)^{p/n},
		\]
		which implies that \( T \) must be normal.
		
		\medskip 
		
		In particular, Theorem \ref{thm:stampfli_thm} also applies when hyponormality is replaced by quasinormality or subnormality. Here, we present  direct proofs for the quasinormal case, as it may be useful in other settings (see \cite[Theorem 2.10]{DehimiMortad23} for the unbounded case). We begin by recalling a fundamental characterization of quasinormal operators.

		\begin{theorem}
			\cite{Embry73}\label{thm:embry_char}
			Let $T$ be a bounded operator on $\mathcal{H}$. Then the
			following conditions are equivalent:
			\begin{enumerate}[(i)]
				\item $T$ is quasinormal;
				\item $(T^*)^kT^k=(T^*T)^k$, for each $k\in\mathbb{N}$;
				\item there exists a (unique) spectral Borel measure $E$ on $\mathbb{R}_+=[0,\infty)$ such that
				$$(T^*)^kT^k=\int_{\mathbb{R}_+}x^k\, E(dx),\quad k\in\mathbb{N}.$$
			\end{enumerate}
		\end{theorem}

        A direct proof for the quasinormal case goes as follows.
		
		\begin{theorem}\label{thm:quasinormal_nth_roots}
			Let $T\in\mathfrak{B}(\mathcal{H})$ be a quasinormal operator. If $T^n$ is normal for some $n\in\mathbb{N}$, then $T$ is normal.
		\end{theorem}
		
		\begin{proof}
			Assume that $T^n$ is normal for some $n\in\mathbb{N}$. Since $T^n$ commutes with $T$, by Fuglede-Putnam theorem we have that $T^n$ also commutes with $T^*$, and thus, it commutes with ${T^*}^{n-1}$. Now, using Theorem \ref{thm:embry_char}, we have that
			\begin{equation*}
				T^*T(T^*T)^{n-1}=(T^*T)^n={T^*}^nT^n=T^*T^{n}{T^*}^{n-1}=T^*T(T^{n-1}{T^*}^{n-1}).
			\end{equation*} 
			From here,
			\begin{equation}\label{eq:quasi_zero}
				T^*T\left((T^*T)^{n-1}-T^{n-1}{T^*}^{n-1}\right)=0.
			\end{equation}
			By applying Theorem \ref{thm:embry_char} once again, and taking the adjoint on both sides of \eqref{eq:quasi_zero}, we have that 
			\begin{equation*}
				\left({T^*}^{n-1}T^{n-1}-T^{n-1}{T^*}^{n-1}\right)T^*T=0.
			\end{equation*}
		It  follows that ${T^*}^{n-1}T^{n-1}=T^{n-1}{T^*}^{n-1}$ on $\overline{\mathcal{R}(T^*)}$. Also, since $\mathcal{N}(T)\subseteq\mathcal{N}(T^*)$ (because $T$ is quasinormal, and thus hyponormal), it is clear that ${T^*}^{n-1}T^{n-1}=T^{n-1}{T^*}^{n-1}$ on 	$\mathcal{N}(T)$. Thus, ${T^*}^{n-1}T^{n-1}=T^{n-1}{T^*}^{n-1}$ on $\mathcal{H}$, demonstrating that $T^{n-1}$ is normal. By continuing this process, we may conclude that $T$ is normal.	\end{proof}

		Another, even simpler proof, can be derived from the following observation.
		\begin{lemma}\label{lem:kernel_inclusion}
			Let $T\in\mathfrak{B}(\mathcal{H})$ be a quasinormal operator. If $\mathcal{N}(T^*)\subseteq\mathcal{N}(T)$, then $T$ is normal.
		\end{lemma}
		\begin{proof}
			By multiplying the equality $TT^*T=T^*TT$ from the right by $T^*$, we have that $$TT^*TT^*=T^*TTT^*,$$ and thus $TT^*=T^*T$ on $\overline{\mathcal{R}(T)}$. Furthermore, $\mathcal{N}(T^*)\subseteq\mathcal{N}(T)$ clearly implies that $TT^*=T^*T$ on $\mathcal{N}(T^*)$, proving that $T$ is normal.
		\end{proof}		
		
		\begin{proof}[A second proof of Theorem \ref{thm:quasinormal_nth_roots}]
			Due to Theorem \ref{thm:embry_char} and the normality of $T^n$, we have that
			\begin{equation*}
				(T^*T)^n={T^*}^nT^n=T^n{T^*}^n.
			\end{equation*}
			From here,
			\begin{equation*}
				\mathcal{N}(T^*)\subseteq\mathcal{N}({T^*}^n)=\mathcal{N}(T^n{T^*}^n)=\mathcal{N}((T^*T)^n)=\mathcal{N}(T^*T)=\mathcal{N}(T).
			\end{equation*}
			The conclusion now follows from Lemma \ref{lem:kernel_inclusion}.
		\end{proof}
		
		\medskip 
		
		Although it predates the work of Aluthge and Wang \cite{AluthgeWang99}, the following result by Ando \cite{Ando72} extends Theorem \ref{thm:stampfli_thm} to a broader class of operators—namely, the class of paranormal operators.

		\begin{theorem}\cite[Theorem 6]{Ando72}\label{thm:ando_thm}
			Let $T\in\mathfrak{B}(\mathcal{H})$ be a paranormal operator. If there exists $n\in\mathbb{N}$ such that $T^n$ is normal, then $T$ is normal.
		\end{theorem}
		The following example demonstrates that this result does not extend further along the inclusion chain \eqref{eq:inclusion_chain}.
		\begin{example}
			Let $M,N\in\mathfrak{B}(\mathcal{H})\setminus\{0\}$ be such that $M$ is normal, $N^2=0$, and $\norm{N}\leq \norm{M}$. Define $T\in\mathfrak{B}(\mathcal{H}\oplus\mathcal{H})$ as
			\begin{equation*}
				T=\begin{bmatrix}
					M&0\\0&N
				\end{bmatrix}.
			\end{equation*}Then, for each $n\geq 2$,
			\begin{equation*}
				\norm{T^n}=\norm{\begin{bmatrix}
						M^n&0\\0&0
				\end{bmatrix}}=\norm{M^n}=\norm{M}^n=\norm{T}^n,
			\end{equation*}
			showing that $T$ is normaloid. Clearly, $T^n$ is normal, while $T$ is not.
		\end{example}

		\medskip 
		
		 While the $n$-th root problem for normal operators has been explored extensively, a parallel investigation for quasinormal operators has only recently begun. In particular, the study of the $n$-th root problem for the class of quasinormal operators was initiated in \cite{CurtoLeeYoon20}, where the authors considered subnormal square roots of quasinormal operators.

As a simple observation, if \( T \) is subnormal (or even paranormal) and right invertible, and \( T^n \) is quasinormal, for some $n\in\mathbb{N}$, then \( T \) must be normal. Indeed, quasinormality of \( T^n \) implies
\[
({T^*}^n T^n - T^n {T^*}^n) T^n = 0,
\]
and since \( T \) is right invertible, it follows that \( T^n \) is normal.  Consequently, \( T \) is normal (by Theorem \ref{thm:ando_thm}). 

Looking from the other side,  the authors in \cite{CurtoLeeYoon20} established the following result.
		
		\begin{theorem}\cite[Theorem 2.3]{CurtoLeeYoon20}
			Let \( T \in \mathfrak{B}(\mathcal{H}) \) be a left-invertible subnormal operator. If \( T^2 \) is quasinormal, then \( T \) is quasinormal.
		\end{theorem}
		
		Soon after, Pietrzycki and Stochel \cite{PietrzyckiStochel21} removed the left-invertibility assumption and generalized the result to arbitrary powers:
		
		\begin{theorem}\cite[Theorem 1.2]{PietrzyckiStochel21}\label{thm:nth_root_quasi_sub}
			Let \( T \in \mathfrak{B}(\mathcal{H}) \) be a subnormal operator such that \( T^n \) is quasinormal for some \( n \in \mathbb{N} \). Then \( T \) is quasinormal.
		\end{theorem}
		
		For an elementary proof of Theorem \ref{thm:nth_root_quasi_sub}, see \cite{Stankovic23_roots}. Additional related results can be found in \cite{Stankovic23_factors}.
		
		In a subsequent paper, Pietrzycki and Stochel extended Theorem \ref{thm:nth_root_quasi_sub} further along the inclusion chain \eqref{eq:inclusion_chain}, proving the following result:
		
		\begin{theorem}\cite[Theorem 1.2]{PietrzyckiStochel23}\label{thm:nth_root_quasi_A}
			Let \( T \in \mathfrak{B}(\mathcal{H}) \) be a class \( A \) operator such that \( T^n \) is quasinormal for some \( n \in \mathbb{N} \). Then \( T \) is quasinormal.
		\end{theorem}

		\medskip 
		The  \( n \)-th root problem for the class of subnormal operators notably differs from the previous cases. Specifically, as shown in \cite{Stampfli66}, a hyponormal operator \( T \) may have a power \( T^n \) that is subnormal for some \( n \in \mathbb{N} \), while \( T \) itself is not subnormal. This example indicates that the \( n \)-th root problem for subnormal operators has a negative answer with respect to any class further down the inclusion chain \eqref{eq:inclusion_chain}.

		\medskip 
		
		The previous discussion naturally gives rise to the following two questions of interest. 
		
		\begin{question}\label{q:root_problem}
			Let \( T \in \mathfrak{B}(\mathcal{H}) \) be a \( k \)-paranormal (absolute-\( k \)-paranormal) operator for some \( k \in \mathbb{N} \). If \( T^n \) is normal for some \( n \in \mathbb{N} \), does it follow that \( T \) is normal?
		\end{question} 
		
		\begin{question}
			Let \( T \in \mathfrak{B}(\mathcal{H}) \) be a paranormal operator. If \( T^n \) is quasinormal for some \( n \in \mathbb{N} \), does it follow that \( T \) is  quasinormal?
		\end{question} 
		Of course, similar questions can be formulated for other operator classes within the extended chains \eqref{eq:k-para_chain} and \eqref{eq:abs-k-para_chain}.
		
		\bigskip 
		
		\section{Extensions of Ando's Theorem}
		
		\bigskip 
		
		In this section, we provide a positive answer to the Question \ref{q:root_problem}. Specifically, we establish the following results:
		
		\begin{theorem}\label{thm:main_k_paranormal}
			Let $T\in\mathfrak{B}(\mathcal{H})$ be a $k$-paranormal or absolute-$k$-paranormal operator for some $k\in\mathbb{N}$. If there exists $n\in\mathbb{N}$ such that $T^n$ is normal, then $T$ is normal.
		\end{theorem}
		
		\begin{theorem}\label{thm:main_k_quasi}
			Let $\mathcal{H}$ be a separable Hilbert space and let $T\in\mathfrak{B}(\mathcal{H})$ be a $k$-quasi-paranormal operator for some $k\in\mathbb{N}$. If there exists $n\in\mathbb{N}$ such that $T^n$ is normal, then $T = T'\oplus T''$, where $T'$ is normal and $T''$ is nilpotent of index at most $\min\{n, k+1\}$. (Either summand may be absent.)
		\end{theorem}
		
		Due to the additional technical challenges involved in the proof of Theorem \ref{thm:main_k_quasi}, we will present it in full detail. The proof of Theorem \ref{thm:main_k_paranormal}, on the other hand, follows similar ideas; we will therefore highlight only the key differences.
		
		\medskip 
		
		Before proceeding with the proof of Theorem \ref{thm:main_k_quasi}, we refer the reader to \cite{Nielsen80, Schwartz67} for details on the direct integral decomposition of von Neumann algebras, which will play a crucial role in our argument. We also recall the following characterization of $k$-quasi-paranormal operators.
		
		\begin{theorem}\cite[Theorem 2.1]{GaoLi14}\label{thm:characterization_ineq}
			Let \( T \in \mathfrak{B}(\mathcal{H}) \) and $k\in\mathbb{N}$. Then \( T \) is \( k \)-quasi-paranormal if and only if  
			\begin{equation}\label{eq:charterization_ineq}
				T^{*{k+2}}T^{k+2} - 2\zeta T^{*{k+1}}T^{k+1} + \zeta^2 T^{*k}T^k \geq 0 
			\end{equation}
			for all \( \zeta > 0 \).
		\end{theorem}
		
		\begin{proof}[Proof of Theorem \ref{thm:main_k_quasi}]
			Without loss of generality, we may assume that $T^n$ is normal for some $n\geq k$. Let $\mathcal{A}$ denote the abelian von Neumann algebra generated by $T^n$.   By the Spectral Theorem, the Hilbert space $\mathcal{H}$ can be identified with a direct integral
			\[
			\mathcal{H} = \int_{\sigma(T^n)}^\oplus \mathcal{H}_\lambda \, d\mu(\lambda),
			\]
			where each operator in $\mathcal{A}$ acts as a multiplication operator on this decomposition. Specially, 
			\begin{equation*}
				T^n = \int_{\sigma(T^n)}^\oplus \lambda \, d\mu(\lambda).
			\end{equation*}
			Furthermore, the commutant $\mathcal{A}'$ of $\mathcal{A}$ is decomposable relative to this representation. Since $T \in \mathcal{A}'$, it follows that $T$ can be expressed as
			\[
			T = \int_{\sigma(T^n)}^\oplus T_\lambda \, d\mu(\lambda),
			\]
			where $T_\lambda$ is an operator on $\mathcal{H}_\lambda$ for all $\lambda\in\sigma(T^n)$. 
			Therefore, we have that 
			\begin{equation}\label{eq:power_lambda}
				T_\lambda^n = \lambda I_\lambda,
			\end{equation}
			for $\mu$-almost all $\lambda\in\sigma(T^n)$, 
			where $I_\lambda$ denotes the identity operator on Hilbert space $\mathcal{H}_\lambda$, $\lambda\in\sigma(T^n)$. 
			
			Now let $\zeta>0$ be arbitrary. Since $T$ is $k$-quasi-paranormal, it satisfies \eqref{eq:charterization_ineq}, and thus,
			\begin{equation*}
				\int_{\sigma(T^n)}^\oplus\left[T_\lambda^{*{k+2}}T_\lambda^{k+2} - 2\zeta T_\lambda^{*{k+1}}T_\lambda^{k+1} + \zeta^2 T_\lambda^{*k}T_\lambda^k\right]\,d\mu(\lambda)\geq 0.
			\end{equation*}
			By \cite[Proposition 2.6.1]{Nielsen80}, we have that
			\begin{equation}\label{eq:characterization_pointwise}
				T_\lambda^{*{k+2}}T_\lambda^{k+2} - 2\zeta T_\lambda^{*{k+1}}T_\lambda^{k+1} + \zeta^2 T_\lambda^{*k}T_\lambda^k\geq 0,\quad \zeta>0,
			\end{equation}
			for $\mu$-almost all $\lambda\in\sigma(T^n)$. In other words, $T_\lambda$ is $k$-quasi-paranormal for $\mu$-almost all $\lambda\in\sigma(T^n)$, by Theorem \ref{thm:characterization_ineq}.
			
			Let $E$ be a measure zero set such that \eqref{eq:power_lambda} and \eqref{eq:characterization_pointwise} holds on $\sigma(T^n)\setminus E$, and let $\lambda\in \sigma(T^n)\setminus E$ be arbitrary. If $\lambda=0$, then, using \cite[Theorem 2.4]{GaoLi14}, we have that $T_0^{k+1}=0$. Hence, assume that $\lambda\neq 0$. Since $n\geq k$, we have that \eqref{eq:characterization_pointwise} holds for $k=n$, which together with \eqref{eq:power_lambda} implies that
			\begin{equation*}
				|\lambda|^2T_\lambda^{*2}T_\lambda^2-2\zeta |\lambda|^2T_\lambda^*T_\lambda+\zeta^2|\lambda|^2\geq 0,\quad \zeta>0.
			\end{equation*}
			By dividing with $|\lambda|^2\neq 0$, it follows that
			\begin{equation*}\label{eq:paranormal_ineq_pointwise}
				T_\lambda^{*2}T_\lambda^2-2\zeta  T_\lambda^*T_\lambda+\zeta^2 \geq 0,\quad \zeta>0,
			\end{equation*}
			i.e., $T_\lambda$ is paranormal. Let $x_\lambda\in\mathcal{H}_\lambda$ be an arbitrary unit vector. By \cite[Theorem D]{Ito99}, we have that 
			\begin{equation*}
				\norm{T_\lambda x_\lambda}\leq \norm{T_\lambda^2x_\lambda}^\frac{1}{2}\leq \ldots\leq \norm{T_\lambda^nx_\lambda}^\frac{1}{n}\leq \norm{T_\lambda^{n+1}x_\lambda}^\frac{1}{n+1}\leq ...
			\end{equation*}
			In particular, $\norm{T_\lambda x_\lambda}\leq \norm{T_\lambda^nx_\lambda}^\frac{1}{n}$, from where it follows that
			\begin{equation}\label{eq:first_ineq}
				\norm{T_\lambda x_\lambda}\leq |\lambda|^\frac{1}{n},
			\end{equation}
			while $\norm{T_\lambda^nx_\lambda}^\frac{1}{n}\leq \norm{T_\lambda^{n+1}x_\lambda}^\frac{1}{n+1}$ yield
			\begin{equation*}
				|\lambda|^\frac{1}{n}\leq |\lambda|^\frac{1}{n+1}\norm{T_\lambda x_\lambda}^\frac{1}{n+1},
			\end{equation*}
			i.e,
			\begin{equation}\label{eq:second_ineq}
				|\lambda|^\frac{1}{n}\leq \norm{T_\lambda x_\lambda}.
			\end{equation}
			Combining \eqref{eq:first_ineq} and \eqref{eq:second_ineq}, we obtain that
			\begin{equation*}
				T_\lambda^*T_\lambda=|\lambda|^\frac{2}{n}I,
			\end{equation*}
			and consequently, $T_\lambda$ is quasinormal. Since $T_\lambda$ is invertible, we finally conclude that it must be normal. 
			
			Thus, keeping in mind \cite[Proposition 2.6.1]{Nielsen80}, we have that the desired decomposition is $T=T'\oplus T''$, where
			\begin{equation*}
				T'=\int_{\sigma(T^n)\setminus\{0\}}^\oplus T_\lambda\,d\mu(\lambda)\quad 
				\text{ and }\quad T''=T_0.
			\end{equation*}
			This completes the proof.
		\end{proof}
		
		The nilpotent part $T''$ in Theorem \ref{thm:main_k_quasi} can be further elaborated in
		case~of~${n=2}$.
		The crucial result is the following representation theorem of the square roots of normal operators proved by Radjavi and Rosenthal in \cite{RadjaviRosenthal71}. We present it here in a slightly different form.
		\begin{theorem}\cite{RadjaviRosenthal71}\label{thm:normal_square}
			Let $T\in\mathfrak{B}(\mathcal{H})$. The operator $T$ is a square root of a normal operator if and only if 
			\begin{equation*}\label{eq:normal_square_representation}
				T=A\oplus \begin{bmatrix}
					B&C\\
					0&-B
				\end{bmatrix},
			\end{equation*}
			where $A, B$ are normal, $C\geq 0$, $C$ is one-to-one and $BC=CB$. 
			%Moreover, $B$ can be chosen so that $\sigma(B)$ lies in the closed upper half-plane and  the Hermitian part of $B$ is non-negative.
		\end{theorem}
		For other related results regarding the structure of the $n$-th roots of normal operators, we refer the reader to \cite{Duggal93, DuggalKim20, Kerchy95, Kurepa62, Putnam57}.
		
		\begin{theorem}
			Let $\mathcal{H}$ be a separable Hilbert space and let $T\in\mathfrak{B}(\mathcal{H})$ be a $k$-quasi-paranormal operator for some $k\in\mathbb{N}$ such that $T^2$ is normal. Con\-sider the decomposition in Theorem \ref{thm:main_k_quasi}. If the nilpotent part $T''$ is pure, then it is given by 
			\begin{equation}\label{eq:nilpotent_2_rep}
				T''= \begin{bmatrix}
					0&C\\
					0&0
				\end{bmatrix},
			\end{equation}
			with $C$ being an injective positive operator.
		\end{theorem}
		
		\begin{proof}
			According to Theorem \ref{thm:normal_square}, we have that $$	T''= \begin{bmatrix}
				B&C\\
				0&-B
			\end{bmatrix},$$ where $C$ is a positive operator commuting with $B$, and $B$ is normal. By Theorem \ref{thm:main_k_quasi}, $T''$ is nilpotent of index at most 2. In other words,
			$$	{T''}^2= \begin{bmatrix}
				B^2&0\\
				0&B^2
			\end{bmatrix}=0,$$
			from where it follows that $B^2=0$. Since the  only nilpotent normal operator is the zero operator, we conclude that $B=0$, showing that $T''$ is given by \eqref{eq:nilpotent_2_rep}.
		\end{proof}
		
		\medskip 
		
		In order to prove Theorem \ref{thm:main_k_paranormal}, we first need the following two results. 
		
		\begin{lemma}\label{lem:k-char}
			Let $T\in\mathfrak{B}(\mathcal{H})$ and $k\in\mathbb{N}$. 
			\begin{enumerate}[$(i)$]
				\item  $T$ is a $k$-paranormal if and only if
				\begin{equation}\label{eq:k-char_ineq}
					T^{*k+1}T^{k+1}-(k+1)\lambda^k T^*T+k\lambda^{k+1}\geq 0,
				\end{equation}
				for each $\lambda>0$.
				\item  $T$ is a absolute-$k$-paranormal if and only if
				\begin{equation}\label{eq:k-char_ineq}
					T^*(T^*T)^kT-(k+1)\lambda^k T^*T+k\lambda^{k+1}\geq 0,
				\end{equation}
				for each $\lambda>0$.
			\end{enumerate}
		\end{lemma}
		\begin{proof} The proof of the second part can be found in \cite{FurutaItoYamazaki98}. Thus, we only prove $(i)$. 
			
			Let $x\in\mathcal{H}$ and $\lambda>0$ be arbitrary. From \eqref{eq:k-para_def}, we have that
			\begin{align*}
				\skal{T^*Tx}{x}&\leq \skal{T^{*k+1}T^{k+1}x}{x}^\frac{1}{k+1}\skal{x}{x}^{\frac{k}{k+1}}\\
				&=\left[\left(\frac{1}{\lambda}\right)^k\skal{T^{*k+1}T^{k+1}x}{x}\right]^\frac{1}{k+1}\left[\lambda\skal{x}{x}\right]^{\frac{k}{k+1}}.
			\end{align*}  Using Young's inequality for real numbers, this further implies
			\begin{equation}\label{eq:k-char_x}
				\skal{T^*Tx}{x}\leq \frac{1}{k+1}\frac{1}{\lambda^k}\skal{T^{*k+1}T^{k+1}x}{x}+\frac{k}{k+1}\lambda\skal{x}{x},
			\end{equation}
			and thus,
			\begin{equation*}
				T^{*k+1}T^{k+1}-(k+1)\lambda^k T^*T+k\lambda^{k+1}\geq 0.
			\end{equation*}
			By taking $\lambda=\left(\dfrac{\skal{T^{*k+1}T^{k+1}x}{x}}{\skal{x}{x}}\right)^\frac{1}{k+1}$ in \eqref{eq:k-char_x} (if $\skal{T^{*k+1}T^{k+1}x}{x}=0$, we let $\lambda\to 0$), we see that \eqref{eq:k-para_def} holds.  
		\end{proof}

		\begin{lemma}\label{lem:root_of_scalar}
			Let $T\in\mathfrak{B}(\mathcal{H})$ be a $k$-paranormal or absolute-$k$-paranormal operator for some $k\in\mathbb{N}$. If there exists $n\in\mathbb{N}$ and $\lambda\in\mathbb{C}$ such that $T^n=\lambda I$, then $T$ is normal.
		\end{lemma}
		
		\begin{proof}
			Since \(T\) is normaloid in both cases, it follows that
			\[
			\norm{T}^n = \norm{T^n} = |\lambda|,
			\]
			which implies \(\norm{T} = |\lambda|^{\frac{1}{n}}\). If \(\lambda = 0\), the result is immediate. Thus, assume \(\lambda \neq 0\) and let \(x \in \mathcal{H}\) be an arbitrary unit vector. 
			
			First consider the case when $T$ is $k$-paranormal.
			By taking $T^{n-1}x$ instead of $x$ in \eqref{eq:k-para_def}, we get
			\begin{align*}
				\aps{\lambda}^{k+1}&=\norm{T^nx}^{k+1}\leq\norm{T^{n+k}x}\norm{T^{n-1}x}^k\\
				&=\aps{\lambda}\norm{T^{k}x}\aps{\lambda}^{\frac{k(n-1)}{n}}\\
				&\leq \aps{\lambda}^{k+1-\frac{k}{n}}\norm{T}^{k-1}\norm{Tx}\\
				&=\aps{\lambda}^{k+1-\frac{1}{n}}\norm{Tx}\\
				&\leq \aps{\lambda}^{k+1},
			\end{align*}
			and thus, $\norm{Tx}=\aps{\lambda}^{\frac{1}{n}}$. Therefore, $T^*T=\aps{\lambda}^{\frac{2}{n}}$. 
			
			Now assume that $T$ is absolute-$k$-paranormal. In a similar fashion, using \eqref{eq:abs-k-para_def}, we have
			
			\begin{align*}
				\aps{\lambda}^{k+1}&=\norm{T^nx}^{k+1}\leq\norm{|T|^kT^nx}\norm{T^{n-1}x}^k\\
				&=\aps{\lambda}\norm{|T^{k}|x}\aps{\lambda}^{\frac{k(n-1)}{n}}\\
				&=\aps{\lambda}\norm{T^{k}x}\aps{\lambda}^{\frac{k(n-1)}{n}}\\
				&\leq \aps{\lambda}^{k+1},
			\end{align*}
			and again, $T^*T=\aps{\lambda}^{\frac{2}{n}}$.
			
			In both cases, by multiplying this equality from the right by \(T^{n-1}\), we obtain
			\[
			T^* = |\lambda|^{\frac{2}{n}} \lambda^{-1} T^{n-1}.
			\]
			This clearly implies the normality of \(T\).
		\end{proof}

		\begin{proof}[Proof of Theorem \ref{thm:main_k_paranormal}]
			We only present a proof in the case when $T$ is $k$-paranormal, as the second case can be proved analogously. 
			
			First assume that the Hilbert space $\mathcal{H}$ is separable and let us use the same notation as in the proof of Theorem \ref{thm:main_k_quasi}. Let $\zeta>0$ be arbitrary. Since $T$ is $k$-paranormal, it satisfies \eqref{eq:k-char_ineq}, and thus,
			\begin{equation*}
				\int_{\sigma(T^n)}^\oplus\left[T_\lambda^{*k+1}T_\lambda^{k+1}-(k+1)\zeta^k T_\lambda^*T_\lambda+k\zeta^{k+1}\right]\,d\mu(\lambda)\geq 0.
			\end{equation*}
			By \cite[Proposition 2.6.1]{Nielsen80}, we have that
			\begin{equation*} 
				T_\lambda^{*k+1}T_\lambda^{k+1}-(k+1)\zeta^k T_\lambda^*T_\lambda+k\zeta^{k+1}\geq 0,\quad \zeta>0,
			\end{equation*}
			for $\mu$-almost all $\lambda\in\sigma(T^n)$. In other words, $T_\lambda$ is $k$-paranormal for $\mu$-almost all $\lambda\in\sigma(T^n)$, by Lemma \ref{lem:k-char}. Lemma \ref{lem:root_of_scalar} now implies that $T_\lambda$ is normal for $\mu$-almost all $\lambda\in\sigma(T^n)$. This shows that $T$ is also normal.
			
			We now turn to the general case where $\mathcal{H}$ is not necessarily separable. Our approach is inspired by techniques from \cite[Lemma 3.1]{Pietrzycki18}. Let $x\in\mathcal{H}$ be arbitrary. Define $\mathcal{H}_x\subseteq\mathcal{H}$ as
			\begin{equation*}
				\mathcal{H}_x:=\bigvee\left\{{A^*}^{i_k}A^{j_k}\cdots{A^*}^{i_1}A^{j_1}x:\, (i_1,\ldots,i_k),(j_1,\ldots,j_k)\in\mathbb{N}_0^k,\, k\in\mathbb{N}\right\},
			\end{equation*}
			where $\bigvee \mathcal{S}$ denotes the closure of the linear span of vectors in a set $\mathcal{S}\subseteq\mathcal{H}$. Then, it is easy to see that $\mathcal{H}_x$ is a separable Hilbert space and it reduces $T$. Let $T_x:=T\restriction_{\mathcal{H}_x}$. Note that $T_x$ satisfies \eqref{eq:quasi_paranormal_def} for all $y\in\mathcal{H}_x$. Thus, $T_x$ is $k$-quasi-paranormal and $T_x^n$ is normal. Consequently, $T_x$ is normal, and 
			\begin{equation*}
				TT^*x=T_xT^*_xx=T^*_xT_xx=T^*Tx. 
			\end{equation*}
			Since $x\in\mathcal{H}$ was arbitrary, we conclude that $TT^*=T^*T$, i.e., $T$ is normal.
		\end{proof}

		\begin{remark}
			It is clear that Theorem \ref{thm:ando_thm} appears as a special case of Theorem \ref{thm:main_k_paranormal} when $k=1$. Also, the examination of the previous proof naturally leads to the question of whether the separability assumption in Theorem \ref{thm:main_k_quasi} can be dropped.
		\end{remark}

		We conclude the paper with the following result concerning normal powers of operators.

		\begin{corollary}\label{cor:coprime}
			Let $T\in\mathfrak{B}(\mathcal{H})$ be an invertible operator and let $k\in\mathbb{N}$. If $T^m$ is $k$-paranormal or absolute-$k$-paranormal, and $T^n$ is normal, where $m,n\geq 2$ are coprime integers, then $T$ is normal. 
		\end{corollary}
		
		\begin{proof}
			Since $T^n$ is normal, we have that $(T^m)^n=(T^n)^m$ is also normal. By Theorem \ref{thm:main_k_paranormal}, we conclude that $T^m$ is normal. The result now follows from \cite[Theorem 2.21]{DehimiMortad23}.
		\end{proof}

		\bigskip 
		\section*{Declarations}
		%=====================
		\noindent{\bf{Funding}}\\
		This work has been supported by the Ministry of Science, Technological Development and Innovation of the Republic of Serbia [Grant Number: 451-03-137/2025-03/200102].		
		
		%=====================
		\vspace{0.5cm}
		
		%=====================
		\noindent{\bf{Availability of data and materials}}\\
		\noindent No data were used to support this study.
		%=====================
		\vspace{0.5cm}\\
		%=====================
		\noindent{\bf{Competing interests}}\\
		\noindent The authors declare that they have no competing interests.
		%=====================
		\vspace{0.5cm}
		
		%=====================
		\noindent{\bf{Author contribution}}\\
		\noindent
		The work was a collaborative effort of all authors, who contributed equally to writing the article. All authors have read and approved the final manuscript.
		%=====================
		
		\vspace{0.5cm}
		
		%	\noindent{\bf Acknowledgment:} 
		
		%-----------------------------------------------{chapter}{Bibliography}------------------------------------------
		
		%-----------------------------------------------------------------------------------------
	\end{document}